\documentclass[12pt]{amsart}
\usepackage[utf8]{inputenc}
\usepackage{amsmath,amssymb,amsthm}
\usepackage{tikz,tikz-cd} 

\DeclareMathOperator{\Pic}{\mathrm{Pic}}

\newtheorem{conjecture}[equation]{Conjecture} 
\newtheorem{prop}[equation]{Proposition} 
\newtheorem{theorem}[equation]{Theorem}

\newtheorem{lemma}[equation]{Lemma}
\newtheorem{corollary}[equation]{Corollary}

\theoremstyle{remark}
\newtheorem{remark}[equation]{Remark}

\begin{document}
\begin{abstract}
 Serrrano's Conjecture says that if $L$ is a strictly nef line bundle on a smooth projective variety $X$, then $K_X+tL$ is ample for all $ t > \dim X+1$. For nonuniruled varieties, it is a special case of the generalized abundance conjecture of Lazic and Peternell \cite{LP}. In this short note, we show that generalized abundance holds for smooth Fano fibrations over varieties of general type. I also prove a generalized version of this conjecture (due to Campana, Chen and Peternell) for surfaces. We conclude this paper by showing that a conjecture of Segre, Harbourne, Gimigliano and Hirschowitz on the expected dimension of linear systems on blow-ups of $\mathbb{P}^2$ implies the existence of strictly nef non ample divisors on rational surfaces
\end{abstract}

\title[Strictly nef divisors]{Strictly nef divisors and some remarks on a conjecture of Serrano}
\author{Priyankur Chaudhuri}
\address{Department of Mathematics, University of Maryland, College Park, MD 20742}
\email{pchaudhu@umd.edu}

\maketitle

\section{Introduction}
A Cartier divisor $D$ on a projective variety $X$ is called \emph{strictly nef} if $D\cdot C>0$ for all effective curves $C \subset X$. Early on, there were only two known examples of such divisors which are not ample: one due to Mumford of a divisor on a ruled surface which has Iitaka dimenion $-\infty$ and one due to Ramanujam of a divisor on a threefold which is  nef and big. Afterwards, Mumford's example was generalized to higher dimensions by Subramanian (\cite{S}) to construct such divisors on projective bundles of arbitrary rank over curves. Later on, Mehta and Subramanian (\cite{MS}) constructed such examples in finite characteristic.
\\

If $X$ is a projective variety of dimension $d$ and $D$ is any strictly nef divisor on $X$, then the Cone Theorem shows that the adjoint divisor $K_X+tD$ is again strictly nef for all $ t> d+1$ if $X$ is smooth  and for all $ t> 2d$ if $X$ has log terminal singularities (see Lemma 6). It is natural to wonder about the ampleness of these adjoint divisors. In \cite{Ser}, Serrano conjectures that in fact:\\

\begin{conjecture}If $X$ is a smooth projective d-dimensional variety and $D \in \Pic X$ is strictly nef, then $K_X+tD$ is ample for all $ t >d+1$.
\end{conjecture}

Again the Cone Theorem shows that Serrano's Conjecture is equivalent to the assertion that $K_X^{\perp} \cap D^{\perp} = 0$ in $\overline{NE}(X)$.\\

Note that if a strictly nef divisor $L$ is semiample, then it is in fact ample. Indeed, replacing $L$ by a positive tensor power, we may assume that $L$ is globally generated. Let $\phi: X \rightarrow Y$ be the morphism given by the linear system $|L|$. Then $L=\phi^*(A)$ for some ample $A \in \Pic Y$. If $\phi(C)=pt$ for some curve $C \subset X$, then $L|_C=\phi^*(A)|_C =\mathcal{O}_C$ which can not be since $L$ is strictly nef. Thus $\phi$ is finite and $L$, being a finite pullback of an ample divisor is also ample.\\

 Serrano proves his conjecture for Gorenstein surfaces and for most smooth threefolds. His ideas use specific results about classification of surfaces and threefold extremal contractions.  This is followed by Campana, Chen and Peternell \cite{CCP} who prove the conjecture for varieties $X$ with Kodaira dimension atleast $d-2$. The general case seems to be hard and related to the main conjectures of the Minimal Model Program. In fact, for non uniruled varieties, it is a special case of the Generalized Abundance Conjecture (see \cite[page 2]{LP}).
\begin{conjecture}(Generalized Abundance) Let $(X, \Delta)$ be a klt pair with $K_X+\Delta$ pseudoeffective and $L$ a nef Cartier divisor on $X$ such that $K_X+\Delta +L$ is also nef. Then there exists a semiample $\mathbb{Q}$-divisor $M$ on $X$ such that $K_X+\Delta+L \equiv M$.
\end{conjecture}

While this conjecture is expected to fail for uniruled varieties (see Remark 14), the following result produces a large class of uniruled varieties on which generalized abundance holds:

\begin{theorem} Let $X$ be a smooth projective variety admitting a surjective morphism $X \xrightarrow{\phi} Y$ with connected fibers such that 
$-K_X$ is ample on a general fiber of $\phi$ and $Y$ is of general type. Suppose that $L, K_X+L \in \Pic X$ are both nef. Then $K_X+L$ is semiample.
\end{theorem}

In the situation of the above theorem, it is interesting to ask how far one can relax the restriction on $\omega_Y$ being big. For example, can we say anything interesting about $K_X+L$ if $\kappa(Y) \geq 1$?\\

One interesting question about moving strictly nef divisors themselves is: how far are they from being big? While we do not know the answer in general, we show that if $\kappa(L) \geq \dim(X)-2$, then $L$ is big (see Proposition 17). In the fourth section, I prove the surface case of a conjecture of Campana, Chen and Peternell (\cite[Conjecture 2.2]{CCP}), which is a generalization of Serrano's conjecture. Here we briefly recall the result. A line bundle $L \in \Pic X$ is called \emph{almost strictly nef} if there is a birational morphism $ \pi:X \rightarrow Y$ and $M \in \Pic Y$ strictly nef such that $L = \pi^*(M)$.

\begin{theorem} Let $S$ be a smooth projective surface and $ L \in \Pic S$ almost strictly nef. Then $K_S+tL$ is big for all $t>3$.
\end{theorem}
More generally, we show that this even holds for surfaces with Gorenstein singularities (see Corollary 23). As a corollary, we show that this conjecture holds for varieties $X$ with $\kappa(X) \geq \dim(X)-2$ (see Corollary 25). In the fourth section, using a conjecture of Segre, Harbourne, Gimigliano and Hirschowitz (called the SHGH conjecture in short; see \cite[Conjecture 3.1]{Han}) on the expected dimension of ceratin linear systems on blow-ups of $\mathbb{P}^2$, we produce a series of examples of strictly nef non-ample divisors on surfaces of arbitrary Kodaira dimension. More precisely, we show:\\
\begin{theorem}Suppose SHGH conjecture holds. Then for any projective surface $S$, there exists a surface $S^{'}$ birational to $S$ admitting a strictly nef non ample divisor.
\end {theorem}

\subsection{The results we will be using}

The following lemmas are used in the whole paper. The first one is due to Serrano and the second is a modification of one of his results.

\begin{lemma}(\cite[Lemma 1.1]{Ser}) Let $X$ be a projective variety of dimension $d$ and $L\in \Pic X$ strictly nef . Then $K_X+tL$ is strictly nef  for all $ t >d+1$ if $X$ is smooth and  for all $ t>2d$ if $X$ has klt singularities.
\end{lemma}

\begin{proof}
 For the smooth case, see \cite[Lemma 1.1]{Ser}. If $X$ has klt singularities, it follows from the Cone Theorem that $K_X+2dL$ is nef. Thus $K_X+tL$ is strictly nef for all $ t>2d$.
\end{proof}

The following result gives a numerical criterion for generalized abundance to fail. In view of the above lemma, it generalizes \cite[Lemma 1.3]{Ser}.
\begin{lemma}
Let $X$ be a projective variety with at worst klt singularities. Suppose that $L, K_X+L \in \Pic X $ are both nef. If $K_X+L$ is not semiample, then $K_X^d=K_X^{d-1}\cdot L=\dots=L^d=0$ where $d=\dim X$.
\end{lemma}
\begin{proof}
 If $(K_X+2L)^d>0$, then $K_X+2L=2(K_X+L)-K_X$ is nef and big and thus $K_X+L$ is semiample by the basepoint free theorem (see \cite[Theorem 3.3]{KM}) which contradicts our assumption. Thus $(K_X+2L)^d=0$. Since $K_X+L , L$ are both nef, this implies $  (K_X+L)^i\cdot L^{d-i} =0$ for all $i$ and from this, we deduce that $K_X^i\cdot L^{d-i}=0$ for all $i$ .
\end{proof}

\section{Strictly nef divisors and Serrano's Conjecture}

\begin{lemma}Let $X$ be a smooth projective variety of dimension n, and let $L \in \Pic X $ be strictly nef. Assume that $|aK_X+bL|$ contains a smooth divisor $D$ for some $a,b \in \mathbb{Z}$ and that Serrano's conjecture holds for $D$. Then $K_X+tL$ is ample for all $ t > n+1$.
\begin{proof}
Let $D \in |aK_X+bL|$ be smooth. Then $K_D+tL|_D$ is ample for all $ t>n$ by assumption. This implies that
\begin{center}
$L|_D\cdot (K_D+tL|_D)^{n-2} =L\cdot D\cdot (K_X+D+tL)^{n-2}>0$ 
\end{center}
and hence
\begin{center}
$L\cdot (aK_X+bL)\cdot (K_X+aK_X+(b+t)L)^{n-2}>0$
\end{center}
for all $ t>n$. Now we can appeal to Lemma 3. 
\end{proof}
\end{lemma}

\begin{prop} Let $X$ be a smooth projective variety with $n=\dim X \geq 4$ and let $L \in \Pic X$ be strictly nef. Assume that the linear system $ |aK_X+bL|$ is basepoint free with $\kappa (|aK_X+bL|) \geq n-3$ for some $a>0$. Then $K_X+tL$ is ample for all $ t> n+1$.

\begin{proof}
 For the sake of illustration, let us consider a few lower dimensional cases. Suppose $n=4$, $\kappa(|aK_X+bL|) \geq 1$. Let $D \in |aK_X+bL|$ be smooth. Then \begin{center}
$aK_D+bL|_D = a(K_X+D)|_D+bL|_D$
\end{center}
\begin{center}
$=(aK_X+bL)|_D+a(aK_X+bL)|_D$ 
\end{center}
is globally generated. This $ \cong \mathcal{O}_D $ if $\kappa(|aK_X+bL|)=1$. Thus $K_D \equiv _{\mathbb{Q}} \pm L|_D$. Thus $\pm K_D$ is strictly nef, hence ample (by abundance for 3-folds and \cite[Theorem 3.9]{Ser} respectively) and Serrano's Conjecture holds for $D$. If $\kappa(|aK_X+bL|)>1$, then $\kappa(|aK_D+bL|_D|) \geq 1$ and Serrano's Conjecture holds on $D$ by \cite[Proposition 3.1]{Ser}. Now we are done for the case $n=4$ by the above Lemma. \\
 
 Now suppose $n=5$, $\kappa(|aK_X+bL|) \geq 2$, $D \in |aK_X+bL|)$ smooth. Then $|aK_D+bL|_D|$ is basepoint free and \begin{equation}
 \label{5}    
|aK_D+bL|_D| \supset |aK_X+bL||_D + |a(aK_X+bL)||_D \end{equation} and both the pieces have Iitaka dimension  $\geq 1$. Thus $\kappa(|aK_D+bL|_D|) \geq 1$ and we are reduced to the above case.\\

 The general case proceeds by induction as in the above cases once we note that $\kappa (|aK_D+bL|_D|) \geq \kappa (|aK_X+bL|) -1$ for all $ D \in |aK_X+bL|$ smooth by (\ref{5}).

\end{proof}
\end{prop}

\begin{prop} Serrano's conjecture for Calabi-Yau varieties implies Serrano's conjecture for varieties $X$ with $\pm K_X$ semiample.

\begin{proof} Suppose first that $\kappa(\pm K_X) \geq 1$ and $L \in \Pic X$ is strictly nef. Let 
\begin{center}
$\phi=\phi_{|\pm mK_X|} :X \rightarrow Y$
\end{center}
be the Kodaira fibration, so $\pm mK_X =\phi^*(A)$ for some $A \in \Pic Y$ very ample. Let $F \subset X$ be a general fiber of $\phi$. Then 
\begin{center}
$\pm m K_F= \pm mK_X|_F=(\phi^* A)|_F =\mathcal{O}_F$.
\end{center}
Let $f: \hat{F} \rightarrow F$ be an associated unramified cyclic cover (see \cite[Definition 2.49]{KM}). Then we have 
\begin{center}
$f^*(K_F) \cong \mathcal{O}_{\hat{F}} \cong K_{\hat{F}}$.
\end{center}
If $\hat{L}:=f^*(L|_F)$, then $\hat{L}= f^*(K_F+L|_F)$ is ample by assumption. Thus $K_F+L|_F$ is also ample, ie $K_X+L$ is $\phi$-ample. Thus for all $ N \gg 0$, 
\begin{center}
$K_X+L+\phi^*(NA)=(1 \pm mN)K_X+L$
\end{center}
is ample. Hence $((1 \pm mN)K_X+L)^n >0$ and we are done by Lemma 3.\\

If $\kappa(K_X)=0$, then the conclusion follows by using an unramified cyclic cover as above.

\end{proof}
\end{prop}

We have the following amusing proposition:

\begin{prop}
 Let $X = BL_{p}(X')\xrightarrow{\pi}  X'$ be a blow-up  of a smooth projective variety $X^{'}$ with $K_{X^{'}}$ pseudoeffective. Let $L \in \Pic X$ be strictly nef. Then $K_X+tL$ is ample for all $t> n+1$.
\end{prop}
\begin{proof}
 There exists $L^{'} \in \Pic X$ such that \begin{equation} \label{1}
L= \pi ^{*}(L^{'})-bE.
\end{equation}Moreover $b>0$: intersect both sides of (\ref{1}) with an $E$-negative curve. $L'$ is also strictly nef: If $C' \subset X'$ is a curve, then   
\begin{center}
$L'\cdot C' =\pi^*(L')\cdot \pi^*(C')$
\end{center}
\begin{center}
   $= \pi^{*}(L')\cdot ( \Tilde{C'}+mD) =\pi^{*}(L')\cdot \Tilde{C'}>0 $ 
\end{center}
where $ \Tilde{C'}$ is the proper transform of $C$, $D$ is an exceptional curve and $m$ is a non negative integer. If $L^n>0$, then $K_X+tL$ is ample for $t>n+1$ by Lemma 2. If $L^n=0$, then $ L^{'n}>0$ (by (\ref{1}) above). Thus as $K_{X^{'}}$ is pseudoeffective, $K_{X^{'}}+\epsilon L{'}$ is big for all positive $\epsilon$. In particular, 
\begin{center}
$K_X+\epsilon L = \pi^*(K_{X^{'}}+ \epsilon L^{'})+(n-1-b \epsilon)L $,
\end{center} 
being a sum of a big and a nef divisor, is big for all $\epsilon < (n-1)/b$. Thus $K_X+ (t+\epsilon)L $ is big and strictly nef  for all $ t>n+1$ and we are done by Lemma 3.
\end{proof}

 The following simple remark will be used several times in the sequel:\\

 \begin{remark}If $L \in \Pic X$ is nef such that $K_X+L$ is also nef (this happens for example if $L$ is a large multiple of a strictly nef divisor on a smooth projective variety) and $\phi:X \xrightarrow[]{}  Y$ an extremal contraction with some fiber $F$, then $L|_{F}$ is ample: We know that $-K_{X}|_{F}$ is ample and $(K_X+L)|_{F}$ is nef, thus $L|_{F}$ is ample.
 \end{remark}
 
 This allows us to generalize Proposition 11:

\begin{corollary}
Let $X=Bl_p(X^{'}) \xrightarrow{\pi} X^{'}$ be the blow-up of a smooth projective variety $X^{'}$ at a point $p$. Assume that $K_{X^{'}}$ is pseudoeffective and suppose that $L, K_{X}+L \in \Pic X$ are both nef. Then $K_{X}+L$ is semiample.
\end{corollary}

\begin{proof}
 Let $n=\dim X$. Then
$K_X\cdot L^{n-1} = (\pi^*(K_{X^{'}})+E)\cdot L^{n-1}>0$ by Remark 10 and assumption. Thus we are done by Lemma 3. \\
\end{proof}

\begin{remark}
Since for all $ n \geq 1, H^0(nK_X)=H^0(nK_{X^{'}})$ (by \cite[Example 2.1.16]{Laz}), Corollary 10 proves Generalized Abundance for all smooth projective varieties which are obtained by blowing up points on smooth projective varieties $X$ with $K_X$ $\mathbb{Q}$-effective.
\end{remark}

Though Generalized Abundance does not always hold for uniruled varieties (see Remark 14), in the following situation, it does:

\begin{theorem}
 Suppose that $X$ is a smooth projective variety admitting a surjective morphism $X \xrightarrow[]{\phi} Y$, with connected fibers such that $-K_X$ is ample on a general fiber of $\phi$ (for example, $\phi $ could be a Fano contraction) and $Y$ is of general type. Suppose that $L$ and $K_X+L$ are both nef divisors on $X$. Then $K_X+tL$ is big if $t \gg 0$. In particular, $K_X+L$ is semiample.
\begin{proof} There exists a birational morphism $ \tau: Y^{'} \rightarrow Y$, where $Y^{'}$ is smooth projective such that letting $X^{'}$ denote the desingularization of the main component of $X \times _Y Y^{'}$ and $\phi^{'}: X^{'} \rightarrow Y^{'}$, $\tau^{'}: X^{'} \rightarrow X$ the induced morphisms, any $\phi^{'}$-exceptional divisor is also $\tau{'}$-exceptional. Indeed, one can choose $\tau $ to be a birational flattening of $\phi : X \rightarrow Y $. Let $L^{'} := \tau^{'*}(L) $. Then $L^{'}$ is also ample on a general fiber of $\phi^{'}$. Indeed, letting $U \subset Y$ be a non empty open set contained in the smooth locus of $\phi $ over which $-K_X$ is ample and letting $U^{'}:= \tau^{-1} (U)$, $L^{'}$ is ample over $U^{'}$. \\

Then by \cite [Lemma 4.2]{Ser}, $[\phi^{'}_*(\omega_{X^{'}/Y^{'}} \otimes r L^{'}))^{\otimes s}]^{**} \otimes \omega _{Y^{'}} \otimes G^{\otimes (m+1)}$ is generically spanned for all $ r,s >0$, $G \in \Pic Y^{'}$ very ample where $m = \dim Y$. Now since $L^{'}$ is ample on a general fiber of $\phi^{'}$, the natural map 

\begin{center}
$(\phi^{'}_*(\omega_{X^{'}/Y^{'}} \otimes rL^{'}))^{\otimes s} \rightarrow \phi^{'}_*(s(\omega_{X^{'}/Y^{'}} \otimes rL^{'}))$
\end{center}
is generically surjective if $ r \gg 0$ (by combining Kodaira vanishing with \cite[Corollary 2, page 50]{Mum}). Thus we can conclude that 
\begin{center}
$[\phi^{'}_*(s(\omega_{X^{'}/Y^{'}} \otimes rL^{'})]^{**} \otimes \omega_{Y^{'}} \otimes G^{\otimes (m+1)}$
\end{center}
is also generically spanned for all $r,s >0$. By a theorem of Nakayama (see \cite[Theorem 1.2]{Fu}, \cite[Lemma 5.10, page 107]{Nak}), this equals
\begin{center}
$ \phi^{'}_*(s(\omega_{X^{'}/Y^{'}} \otimes rL^{'} \otimes E)) \otimes \omega_{Y^{'}} \otimes G^{\otimes (m+1)}$ 
\end{center}
for some effective $\phi^{'}$-exceptional (hence also $\tau^{'}$-exceptional) divisor $E$. This equals
\begin{center}
$ \phi^{'}_*(s(\omega_{X^{'}} \otimes rL^{'} \otimes E)) \otimes (1-s)\omega_{Y^{'}} \otimes G^{\otimes (m+1)}$.
\end{center}
Now fix $t_0$ sufficiently large such that there exists $D \geq 0$ such that $D \in |t_0K_{Y^{'}}-G|$, so that $ (m+1)(D+G) \in |t_0(m+1)K_{Y^{'}}|$. Then 
\begin{center} 
$\phi^{'}_*(s(\omega_{X^{'}} \otimes rL^{'} \otimes E)) \otimes (1-s) \omega _{Y^{'}} \otimes \mathcal{O}_{Y^{'}}((m+1)(D+G))$
\end{center}
\begin{center}
$ \cong \phi^{'}_*(s(\omega_{X^{'}} \otimes rL^{'} \otimes E)) \otimes (t_0(m+1)+1-s)\omega_{Y^{'}} =: \mathcal{F}$
\end{center}
is also generically spannned. Thus there exists $ 0 \neq \sigma \in H^0(\mathcal{F})$ and $ 0 \neq \tau \in H^0((s-1-t_0(m+1))K_{Y^{'}} \otimes \mathcal {O}_{Y^{'}}(-G))$ for $s \gg 0$ since $K_{Y^{'}}$ is big. Then $ 0 \neq \sigma \otimes \tau  \in H^0(\phi^{'}_*(s(\omega_{X^{'}} \otimes rL^{'} \otimes E)) \otimes \mathcal{O}_{Y^{'}}(-G))$ and thus $ \kappa(s(\omega_{X^{'}} \otimes rL^{'} \otimes E) \otimes \phi^{'*}(-G)) \geq 0$. \\

Now the Easy addition theorem (see \cite[Theorem 3.13, page 51] {Nak}) implies that 
\begin{center}
$\omega_{X^{'}} \otimes rL^{'} \otimes E = \tau^{'*}(\omega_X \otimes rL) \otimes \mathcal{O}_{X^{'}}(E+E^{'})$
\end{center}
 is big for $r \gg 0$ (where $E, E^{'} \geq 0$ are $\tau^{'}$-exceptional divisors). Then 
\begin{center}
$\tau^{'}_*(\tau^{'*}(\omega_X \otimes rL) \otimes \mathcal{O}_{X^{'}}(E+E^{'})) = \omega_X \otimes rL$
\end{center}
is big for $r \gg 0$. Thus, $ (K_X+rL)^n >0$ and by Lemma 3, $K_X+L$ is semiample.
 \end{proof}
\end{theorem}

\begin{remark}
Generalized Abundance can fail without any assumption on $\omega _Y$ as the following example (see \cite [example 1.1]{Sho}) shows :
\end{remark}

Let $C$ be a smooth elliptic curve and let $\mathcal{E}$ be a rank 2 bundle on $C$ given by a non-trivial extension 
\begin{center}
$0 \rightarrow \mathcal{O}_C \xrightarrow{\sigma}\mathcal{E}\rightarrow \mathcal{O}_C \rightarrow 0$
\end{center}
corresponding to a nonzero element $\zeta \in h^1(\mathcal{O}_C)$. Let $X = \mathbb{P}\mathcal {E} \xrightarrow{\pi} C$ be the associated ruled surface, $L= \mathcal{O}_{\mathbb{P}\mathcal{E}}(3)$ and $C_0\subset X$ a section of $\pi$ with $\mathcal{O}_X(C_0)=\mathcal{O}_{\mathbb{P}\mathcal{E}}(1)$. Then, since $K_X= \mathcal{O}_{\mathbb{P}\mathcal{E}}(-2)$, both $L, K_X+L = \mathcal{O}_{\mathbb {P} \mathcal{E}}(1)$ are nef. We will show that $\mathcal{O}_{\mathbb{P}\mathcal{E}}(1)$ is not semiample and that $\kappa (\mathcal {O}_{\mathbb{P}\mathcal{E}}(1))=0$.\\

\begin{proof} We first show that $\mathcal{O}_{\mathbb{P}\mathcal{E}}(1) $ can not be semiample. Suppose it is. Then we can choose $n\gg 0$ such that:

\begin{equation} \label{*} 
 h^0(\mathcal{O}_{\mathbb{P}\mathcal{E}}(n))>2  \end{equation}
 
 and there exists 
\begin{center}
$nC_0 \neq Y \in |\mathcal{O}_{\mathbb{P}\mathcal{E}}(n)|$
\end{center}
which is a smooth elliptic curve. Consider the exact sequence
\begin{equation} \label{11}
0 \rightarrow \mathcal{O}_{\mathbb{P}\mathcal{E}}(n-Y)=\mathcal{O}_{\mathbb{P}\mathcal{E}} \rightarrow \mathcal{O}_{\mathbb{P}\mathcal{E}}(n) \rightarrow \mathcal{O}_Y(n) \rightarrow 0
\end{equation}
on $X$. Now $ C_0^2=0$ implies that $\deg (\mathcal{O}_Y(n)) =0$ and hence $h^0(\mathcal{O}_Y(n)) \leq 1$ which gives $h^0(\mathcal{O}_{\mathbb{P}\mathcal{E}}(n)) \leq 2$ by (\ref{11}) thus contradicting (\ref{*}). This proves that $\mathcal{O}_{\mathbb{P}\mathcal{E}}(1)$ is not semiample .\\
 
 Now we show that $\kappa (\mathcal{O}_{\mathbb{P}\mathcal{E}}(1))=0$. First note that $h^0(\mathcal{E})=1$: this follows from the cohomology exact sequence 
\begin{center}
$0 \rightarrow H^0(\mathcal{O}_C) \rightarrow H^0(\mathcal{E}) \rightarrow H^0(\mathcal{O}_C)\xrightarrow{\phi} H^1(\mathcal{O}_C)$
\end{center}
where the extension class $\phi(1) = \zeta \in H^1(\mathcal{O}_C)$ giving $\mathcal{E}$ being nonzero means that $\phi $ is an isomorphism and thus $h^0(\mathcal{E})=1$. \\
 
 If $\kappa (\mathcal{O}_{\mathbb{P}\mathcal{E}}(1))>0$, choose $n$ to be the smallest integer such that $h^0(\mathcal{O}_{\mathbb{P}\mathcal{E}}(n)) \geq 2$. Thus we have two distinct effective divisors $nC_0, D \in |\mathcal{O}_{\mathbb{P}\mathcal{E}}(n)|$ which can not have any $C_0$ component in common by minimality of $n$. Now $(nC_0\cdot D)=0$ implies that $|\mathcal{O}_{\mathbb{P}\mathcal{E}}(n)|$ is basepoint free, which can not be as we saw above.
\end{proof}

\begin{corollary}
Let $X= Bl_Z(Y) \xrightarrow{\pi} Y$ be the blow-up of a smooth projective variety $Y$ with $\omega _Y \cong \mathcal{O}_Y$ along a smooth subvariety $Z \subset Y$ of general type such that the conormal bundle $F :=\mathcal{N}_{Z/Y}^*$ is pseudoeffective.(i.e. $\mathcal{O}_{\mathbb{P} F}(1)$ is pseudoeffective.) If $L, K_X+L \in \Pic X$ are both nef, then $K_X+L$ is semiample.  
\end{corollary}

\begin{proof}
 Let $E$ denote the exceptional divisor of $\pi$ and let $L \in \Pic X$ be as above. Recall that $\pi|_E: E \rightarrow Z$ can be identified with the natural projection $\mathbb{P}F \rightarrow Z$ under which $E|_E$ identifies with $\mathcal{O}_{\mathbb{P}F}(-1) =: \mathcal{O}_E(-1)$. Let $c$ be the codimension of $Z$ in $Y$. Then 
\begin{center}
$K_E+tL|_E=(K_X+E)|_E+tL|_E$
\end{center}
\begin{center}
$= tL|_E +\mathcal{O}_{E}(-c)$
\end{center}
is nef and big for all $t \gg 0$ by above Proposition. Now $\mathcal{O}_E(1)$ being pseudoeffective forces $L|_E$ to be nef and big. Then if $n= \dim X$, \begin{center}
$E\cdot L^{n-1}=K_X\cdot L^{n-1} >0$
\end{center}
and we are done by Lemma 3.  
\end{proof}

\section{Almost strictly nef divisors}
$L\in \Pic X$ is called \emph{almost strictly nef} (ASN) if there is a birational morphism $ \pi: X \xrightarrow[]{}  Y $ to some projective variety $Y$ and $ M \in \Pic Y$ strictly nef such that $\pi^*M = L$. We have the following interesting property of moving ASN divisors:
\begin{prop} Let $X$ be a projective variety of dimension $n$ and let $L \in \Pic X$ be almost strictly nef (ASN) with $\kappa(L) \geq n-2$. Then $L$ is big.

\begin{proof}
Let the following be a resolution of indeterminacy of the Iitaka fibration $\phi=\phi_{|mL|}$ of $L$:
\begin{center}
\begin{tikzcd} 
X \arrow[r, dotted, "\phi"] & Y\\
\hat{X} \arrow[u,"\pi"] \arrow[ur, "\hat{\phi}"]
\end{tikzcd}
\end{center}

So $ \pi^*(mL)=\hat{\phi}^*(A)+E^{'}$ for some $A \in \Pic Y$ ample and $E^{'}$ effective. Let $\hat{L}=\pi^*L$. Let $F \subset \hat{X}$ be a general fiber of $\hat{\phi}$. Then $\dim F \leq 2$. An effective strictly nef divisor on a surface is ample, so we may assume $n \geq 3$. Then $\hat{L}|_{F}$ is ASN. Consider the exact sequence \begin{center}
$0 \rightarrow m\hat{L}\otimes I_F \rightarrow m\hat{L} \rightarrow m\hat{L}|_F \rightarrow 0$.
\end{center}
If $H^0(m\hat{L}|_F)=0$, then $H^0(m\hat{L}\otimes I_F)=H^0(m\hat{L})$ which can only happen if $F \subset Bs|m\hat{L}|=E^{'}$. This is impossible since $F$ is a general fiber. Thus $m\hat{L}|_F$ is ASN and effective. Since $\dim F \leq 2$, $\hat{L}|_F$ is big. Now $\hat{L}$ being $\hat{\phi}$-big,
\begin{center}
 $\hat{L}+\hat{\phi}^*(A)=(m+1)\hat{L}-E^{'}$
\end{center}
 is big by \cite[lemma 2.5]{CCP}. Thus $\hat{L}$ is also big.

\end{proof}
\end{prop}
\begin{remark}
Thus Serrano's conjecture holds for all strictly nef line bundles $L \in \Pic X$ with $\kappa(L) \geq \dim X -2$ by Lemma 3.
\end{remark}

\begin{remark}
I do not know if it is possible for an ASN $L \in \Pic X$ to have $0 \leq \kappa(L) <\dim (X)-2$.
\end{remark}

Campana, Chen and Peternell conjecture that if $X$ is a smooth projective variety and $L \in \Pic X$ is almost strictly nef, then $ K_X+tL$ is big for all $t> \dim (X)+1$ (\cite[Conjecture 2.2]{CCP}). We prove this conjecture when $X$ is a surface. First, let us record a result we will be using:

\begin{prop}
 Let $ \phi: X \rightarrow Y$ be a surjective morphism with connected fibers between smooth projective varieties . Let $L$ be a nef line bundle on $X$ whose restriction to a general fiber of $\phi$ is ample. Let $\Delta_Y$ be an effective $\mathbb{Q}$ divisor on $Y$ such that $K_Y+\Delta_Y$ is big and let $\Delta_X=\phi^*(\Delta_Y)$. Then there exist positive integers $a,b$ with $b/a > \dim (X)+1$ such that $a(K_X+\Delta_X)+bL$ is linearly equivalent to a non zero effective (integral) divisor.
 
 \begin{proof}
  The argument given in the proof of \cite[Proposition 4.3]{Ser} works if we replace $K_Y$ in the proof with some positive integral multiple of $K_Y+\Delta_Y$.
 \end{proof}
\end{prop}

\begin{theorem}
Let $S$ be a smooth projective surface, $L\in \Pic S$ almost strictly nef. Then $K_S+tL$ is big for all $ t > 3$.
\begin{proof}
 If $\kappa(S)=0$, this is \cite[Proposition 2.3]{CCP}. We will treat the other cases:
 \\
 
 \underline{Case 1} : $\kappa(S)>0$: Let $S \xrightarrow{\pi} S'$ be the minimal model for $S$. Note that if $\pi^*(K_{S^{'}})+tL$ is big , then so is $K_S+tL$ and we are done. Otherwise, if $\pi^*(K_{S^{'}})+tL$ is not big, then $(\pi^*K_{S^{'}}+tL)^2 =0$, thus 
\begin{center}
$\pi^*K_{S^{'}}^2 = \pi^*(K_{S^{'}})\cdot L = L^2 =0$,
\end{center}
which by Hodge Index Theorem, implies that $\pi^*K_{S^{'}}=cL$ for some $c>0$. Since $\pi^*K_{S^{'}}$ is semiample, hence $L$ and $\pi^*K_{S^{'}}$ are both big. Thus $K_S= \pi^*K_{S^{'}}+E$ (where E is some effective divisor) is big. Hence $K_S+tL$ is also big.
 \\
  
  \underline{Case 2} : $\kappa(S)= -\infty$: Suppose we have $\pi: S \rightarrow  S_0$ birational, $S_0$ normal projective such that $L=\pi^*M$, where $M \in \Pic (S_0)$ is strictly nef. Note: we may assume that $S_0$ is singular, because if $S_0$ is smooth, then $K_{S_0}+tM$ is ample for all $ t> 3$ by \cite[Proposition 2.1]{Ser} which would imply that $K_S+tL$ is big and we are done. Then $\pi$ can be factorized as $\pi: S \xrightarrow{p} S^{'} \xrightarrow{f} S_0$ such that $S{'}$ is smooth and $f$ does not contract any (-1)- curves. $S^{'}$ can be constructed as follows: 
let $S_1$ be a smooth surface obtained by contracting some (-1) curve (if any) in $Ex(\pi)$. Repeat this process to the induced morphism $\pi_1 : S_1 \xrightarrow{} S_0$. Since the Picard rank drops at each step, the process eventually ends with $f:S^{'} \xrightarrow{} S_0$ as above. Now 
\begin{center}
$K_S+tL=(p)^*(K_{S^{'}}+tL^{'})+$ some effective divisor,
\end{center}
where $L^{'}=f^*(M)$. Hence it is enough to show that $K_{S^{'}}+tL^{'}$ is big. Thus replacing $S$ with $S{'}$ and $L$ with $f^*(M)$, we may assume that $\pi$ does not contract any (-1) curves. Moreover, since $\kappa(S) = -\infty$, $S$ is obtained from a projective bundle over a smooth curve $C$ by a sequence of blow-ups, giving $\phi:S \rightarrow C$ or $S=\mathbb{P}^2$. The Theorem is clear in the latter case.
  \\
  
  With this in mind, we claim that $K_S+tL$ is nef for all $ t\geq 3$. Indeed, if $D \subset S$ is a curve, then 
\begin{center}
$D \equiv \Sigma_{i=1}^m a_i C_i + N, a_i \geq 0$ 
\end{center}
for all $i, N \in \overline{NE}(S)$ with $(N \cdot K_S) \geq 0$ and $C_i \subset S$ extremal rational curves with $0>(K_S \cdot C_i) \geq -3$ for all $i$. Then by adjunction, $C_i^2 = -1, 0$ or $1$  for all $i$ (This follows from classification of extremal contractions on surfaces, see \cite[Theorem 1-4-8]{Mat}). Since $ \pi:S \rightarrow  S_0$ does not contract any such curves and $L=\pi^*(M)$, $M$ strictly nef, thus $(L\cdot C_i) \geq 1$ for all $i$ and thus $(K_S+tL\cdot C_i) \geq 0$ for all $ t \geq 3$. Therefore, $(K_S+tL\cdot D) \geq 0$ for all $ t \geq 3$, which proves the claim.
  \\
  
  Thus if $K_S+tL$ is not big for some $t>3$, then $(K_S+tL)^2= (K_S+3L+(t-3)L)^2=0$. Since $K_S+3L$ and $L$ are both nef, this means  $(K_S+3L)^2=L^2=0$ , so \begin{equation} \label{3} K_S^2=K_S\cdot L=L^2 =0 \end{equation}\\
  By Hodge Index Theorem, \begin{equation} \label{4} -K_S= L^{\otimes m} \end{equation} for some $0<m \leq 3.$\\
  
  Thus $K_S+tL$ is almost strictly nef  for all $ t>3$. If $t \gg 0$, then $K_S+tL$ is very ample on the fibers of $\phi : S \rightarrow C$. Since $C$ is a smooth curve, $\phi_*(\mathcal{O}_S(K_S+tL))$ being torsion free, is a bundle of rank $>1$. The remainder of the proof will be divided into three cases: \\
  
    \underline{Case 2.a)}: $g(C)>1$: Then by Proposition 18, (setting $\Delta_X= \Delta_Y=0$) 
    $h^0(r(K_S+tL)) \geq 1$ for some $r>0, t>3$ and $r(K_S+tL)$ is almost strictly nef. $(r(K_S+tL))^2=0$ iff $r(K_S+tL)$ is a curve class contracted by $\pi$. Now since birational morphisms between surfaces can only contract curves of negative self-intersection, $(K_S+tL)^2>0$ for all $ t \gg 0$, so $K_S+tL$ is big for all $ t \gg 0$.
    \\
    
    \underline{Case 2.b)}: $g(C)=0$: i.e. $C \cong \mathbb{P}^1$. Then it follows from \cite[Lemma 4.2]{Ser} that 
\begin{center}
$\phi_*((\omega_{S/\mathbb{P}^1} \otimes L^{\otimes N})^{\otimes s})$ 
\end{center}
is generically spanned for all $ s>0$, i.e. 
\begin{center}
$(\phi_*(\omega_S \otimes L^{\otimes N}))\otimes \mathcal{O}_{\mathbb{P}^1}(-2)$ 
\end{center}
and hence $\phi_*((N-m)L)$ is generically spanned. If $N \gg m$, then $\phi_*((N-m)L)$ has rank $>1$ and hence $h^0 >1$ and $L$ is almost strictly nef. Then $L$ is big as above and we done by (\ref{3}).\\
   
 \underline{Case 2.c)}: $g(C)=1$. In this case, C is an elliptic curve and we can argue as in the proof of \cite[Theorem 3.1] {CCP} (the proof given there actually uses only nefness of $L,K_X+tL$ and relative bigness of $K_X+tL$) to show that $h^0(S, a(K_S+tL)\otimes \phi^*P) \neq 0$ for some $a>0$ and $P \in \Pic ^0(C)$. Thus $a(K_S+tL) \equiv D \geq 0$ and it is also nef. Now $D$ is not big iff $D^2=0$ iff $D$ is numerically equivalent to a $\pi$-exceptional curve which can not be since $D$ is nef (recall $-K_S$ and $L$ were positively proportional). Thus $D^2>0$ and we are done as before. This finishes the proof.
\end{proof}
\end{theorem}

\begin{corollary}
Let $S$ be a Gorenstein surface and $L \in \Pic S $ almost strictly nef (ASN). Then $K_S+tL$ is big for all $ t>3$.

\begin{proof}
We follow the arguments of \cite[Theorem 2.3]{Ser}. Consider 
\begin{center}
$\pi: R\xrightarrow{f} T \xrightarrow{j} S$
\end{center}
where $T$ is the normalization of $S$ and $R$ is a desingularization of $T$. Then by Theorem 20, $K_R+t\pi^*L$ is big for all $ t \geq 4$. We use the following result of Kawamata:

\begin{prop}(See \cite[Proposition 7]{Kawa}) Let $X$ and $Y$ be Gorenstein varieties, $j:Y^{'} \rightarrow Y$ the normalization of $Y$ and let $f:X \rightarrow Y^{'}$ be a proper birational morphism. Then for all $ m \in \mathbb{N}$, there exists a natural injective homomorphism $f_*(\omega_X^{\otimes m})\hookrightarrow j^*(\omega_Y^{\otimes m})$.
\end{prop}
This gives us an injection 
\begin{center}
$f_*(r(K_R+m\pi^*L))=f_*(rK_R) \otimes j^*(rmL) \hookrightarrow j^*(rK_S+mrL)$.
\end{center}
Let $m \geq 4$. Then 
\begin{center}
$h^0(r(K_R+m \pi^*L)) \sim r^2$
\end{center}
for all $ r \gg 0$ and thus \begin{center}
$h^0(j^*(rK_S+mrL)) \sim r^2$
\end{center}
 if $r \gg 0 $. Now there exists an exact sequence 
\begin{center}
$0\rightarrow \mathcal{O}_S \rightarrow j_*\mathcal{O}_T \rightarrow \mathcal{F} \rightarrow 0$
\end{center}
where $\mathcal{F}$ is an $\mathcal{O}_S$-module supported on a closed subvariety. Tensoring with $\mathcal{O}_S(r(K_S+mL))$ and taking $H^0$, this gives: 
\begin{center}
$0\rightarrow H^0( \mathcal{O}_S(r(K_S+mL))) \rightarrow H^0(j_*j^*(r(K_S+mL))) \rightarrow H^0(\mathcal{F} \otimes \mathcal{O}_S(r(K_S+mL))) \rightarrow $.
\end{center}
The middle term $\sim r^2$ and the right term $\lesssim r$ for all $ r \gg 0$. Thus 
\begin{center}
$h^0(r(K_S+mL)) \sim r^2$
\end{center}
for all $ r \gg 0$.

\end{proof}
\end{corollary}

\begin{corollary} Let $X$ be a smooth projective variety of dimension $n$ and let $L \in \Pic X$ be almost strictly nef (ASN) with $\kappa (aK_X+bL) \geq n-2$ for some $a \geq 0$. Then $K_X+tL$ is big  for all $ t \gg 0$.
\begin{proof}
The proof is quite similar in spirit to that of Proposition 17 and \cite[Theorem 2.6]{CCP}, so the details are left out. Letting $M=aK_X+bL$, we consider a smooth resolution of indeterminacy $\pi: \hat{X} \rightarrow X $ of the Iitaka fibration $\phi = \phi _{|mM|}$. Let $\hat {\phi}: \hat{X} \rightarrow Y$ be the composite. Let $F \subset \hat{X}$ be a general fiber of $\hat{\phi}$ and $ \hat{L}=\pi^*L$. Suppose $K_{\hat{X}}=\pi^*K_X+E^{'} $ and 
\begin{center}
$\pi^*(mM)=\hat{\phi}^*(A)+E$,
\end{center}
$ A \in \Pic Y$ ample and $E, E^{'} \geq 0$. Then we can show that $K_{\hat{X}}+t\hat{L}$ is $\hat{\phi}$-big by Theorem 19 and pseudoeffective, thus $K_{\hat{X}}+t\hat{L}+\hat{\phi}^*(A)$ is big for all $t \gg 0$ (see \cite[Lemma 2.5]{CCP}) and this will show that $K_{\hat{X}}+t\hat{L}$ is big for all $t \gg 0$. Hence, so is $K_X+tL= \pi_*(K_{\hat{X}}+t \hat{L})$.
\end{proof}
\end{corollary}

\section{Examples of strictly nef non-ample divisors}
In this section, we prove Theorem 5. Our calcultions are in the spirit of \cite[example 3.3]{LR}.\\

First, we need to state the following conjecture (see \cite[Conjecture 3.1]{Han} and the references listed there):

\begin{conjecture}
(SHGH) Fix integers $d,m_1, \dots m_r \geq 0$, $r>0$. Let $X=Bl_{p_1, \dots p_r}(\mathbb{P}^2) \xrightarrow{\pi} \mathbb{P}^2$ be a blow-up of $\mathbb{P}^2$ at $r$ general points and $E_1, \dots , E_r$ the exceptional curves. Suppose the linear system $L:=|d\pi^* \mathcal{O}_{\mathbb{P}^2}(1)-m_1E_1- \dots -m_rE_r|$ contains a reduced curve. Then $\dim L = \binom {d+2} {2} -1-\Sigma _i \binom{m_i+1}{2}$ (The quantity on the right hand side is often called the expected dimension of $L$).
\end{conjecture}

Now fix any integer $d \geq 4$. Let $p_1,\dots,p_{d^2} \in \mathbb{P}^2$ be general points and $X=Bl_{p_1,\dots,p_{d^2}} \mathbb{P}^2 \xrightarrow{\pi} \mathbb{P}^2$
 be the blow-up with exceptional divisors $E_1,\dots,E_{d^2}$. Let $l=\pi^*(\mathcal{O}_{\mathbb{P}^2}(1))$
be the pulled back hyperplane class. Let 
\begin{center}
 $L=dl-E_1- \dots -E_{d^2}$.
\end{center}
Then clearly $L^2=0$ and $L\cdot E_i=1$ for all $ i$. Now $d\geq 4$ and the generality of the $p_i$ ensures that $\kappa(L) = -\infty$. Let $C= ml-\Sigma _{i=1}^{d^2} r_i E_i$ be the proper transform of a curve of degree m having singularities of orders $r_i-1$ at $p_i$. Note that $(m,r_1,\dots,r_{d^2})$ are the coordinates of $C$ on $\Pic (X) \otimes \mathbb{R} \cong \mathbb{R}^{d^2+1}$. Since for testing strict nefness of $L$ we may assume that $C$ is reduced, the SHGH conjecture implies that the expected dimension of the linear system $|mL- \Sigma _{i=1}^{d^2}r_iE_i|$  which is 
\begin{center}
$\binom{m+2}{2}-1-\Sigma _1 ^{d^2} \binom{r_i+1}{2} $
\end{center}
is nonnegative which simplifies to 
\begin{center}
$\rho(m)^2 :=m^2+3m-d^2/4 \geq \Sigma _1 ^{d^2}(r_i+1/2)^2$.
\end{center}
 For $m$ fixed, this is the equation of a sphere with center 
\begin{center}
$A=(-1/2,\dots , -1/2) \in \mathbb{R}^{d^2}$
\end{center}
 and radius 
\begin{center}
$\rho (m)= \sqrt{m^2+3m-d^2/4}$
\end{center}
in $\mathbb{R}^{d^2}$. Now $L\cdot C=dm-\Sigma_1^{d^2} r_i$. For $m$ fixed, consider the hyperplane 
\begin{center}
$H=(r_1+\dots+r_{d^2}-dm=0) \subset \mathbb{R}^{d^2}$. \end{center}
Notice that $A \in H_{<0}$ and 
\begin{center}
$AH= d/2+m =: \beta (m)$.
\end{center}
 In fact, the whole sphere above is contained in $H_{<0}$ because 
\begin{center}
$\rho(m)^2=m^2+3m-d^2/4 \leq \beta (m)^2=m^2+md+d^2/4 $
\end{center} for all $d \geq 4$ and thus $L\cdot C>0$.\\

Now let $S$ be a projective surface. It admits a finite morphism $f: S \xrightarrow{} \mathbb{P}^2$ which, suppose, is of degree $m$. Since the points $p_1,\dots,p_{d^2} \in \mathbb{P}^2$ are general, we may assume that they are not in the branch locus of $f$. Let $ \{q_1,\dots,q_{md^2}\}=f^{-1}\{p_1,\dots,p_{d^2}\}$. Then $f$ extends to a finite morphism $f^{'}:S^{'}:=Bl_{q_1,\dots,q_{md^2}}(S) \rightarrow X$ which is the base change of $f$ via $\pi$. Now $f^{'*}(L) \in \Pic (S^{'})$ is strictly nef and non-ample. 

\section{Acknowledgements}
I thank my advisor Patrick Brosnan for numerous conversations and remarks which helped improve this paper in many ways. Thanks also to Caucher Birkar for telling me about the paper \cite{LP} of Lazić and Peternell and to Krishna Hanumanthu for alerting me about a gap in a previous version of Section 4 and for his comments on the SHGH conjecture.
 
\end{document}